\documentclass[11pt]{amsart}

\makeatletter
\usepackage{amssymb}
\usepackage{latexsym}
\usepackage{amsbsy}
\usepackage{amsfonts}

\def\marginpar#1{\ignorespaces}

\newtheorem{theorem}[equation]{Theorem}
\newtheorem{proposition}[equation]{Proposition}
\newtheorem{lemma}[equation]{Lemma}

\theoremstyle{definition}
\newtheorem{remark}[equation]{Remark}

\numberwithin{equation}{section}

\def\AArm{\fam0 \rm}%
\newdimen\AAdi%
\newbox\AAbo%
\def\AAk#1#2{\setbox\AAbo=\hbox{#2}\AAdi=\wd\AAbo\kern#1\AAdi{}}%

\newcommand{\BBone}{{\ensuremath{{\AArm 1\AAk{-.8}{I}I}}}}

\def\eqref#1{(\ref{#1})}
\def\eqlabel#1{\def\@currentlabel{#1}}

\def\formula#1{\def\@tempa{#1}\let\@tempb\theequation\def\theequation{%
\hbox{#1}}\def\@currentlabel{(\theequation)}$$}
\def\endformula{\leqno\hbox{(\@tempa)}$$\@ignoretrue\let\theequation\@tempb}

\def\given{\hskip5\p@\relax\vrule\@width.4\p@\hskip5\p@\relax}

\newcommand{\open}[1]{%
\par\normalfont\topsep6\p@\@plus6\p@\trivlist\item[\hskip\labelsep\itshape#1%
\@addpunct{.}]\ignorespaces}

\DeclareRobustCommand{\close}[1]{%
  \ifmmode 
  \else \leavevmode\unskip\penalty9999 \hbox{}\nobreak\hfill
  \fi
  \quad\hbox{$#1$}}

\newlength{\toskip}

\def\<{\langle}
\def\>{\rangle}

\def \R {{\mathbb R}}

\def \D {{\mathbb D}}
\def \L {{\mathbb L}}

\def \Var {\textrm{Var}}

\makeatother

\begin{document}
\date{\today}

\title[Poincar\'e and $\L^p$]{Poincar\'e inequality and the $\L^p$ convergence of semi-groups.}

 \author[P. Cattiaux]{\textbf{\quad {Patrick} Cattiaux $^{\spadesuit}$}}
\address{{\bf {Patrick} CATTIAUX},\\ Institut de Math\'ematiques de Toulouse. UMR 5219. \\
Universit\'e Paul Sabatier,
\\ 118 route
de Narbonne, F-31062 Toulouse cedex 09.} \email{cattiaux@math.univ-toulouse.fr}

\author[A. Guillin]{\textbf{\quad {Arnaud} Guillin $^{\diamondsuit}$}}
\address{{\bf {Arnaud} GUILLIN},\\ Laboratoire de Math\'ematiques. UMR 6620, Universit\'e Blaise Pascal,
avenue des Landais, F-63177 Aubi\`ere.} \email{guillin@math.univ-bpclermont.fr}

\author[C. Roberto]{\textbf{\quad {Cyril} Roberto ${\clubsuit}$}}
\address{{\bf {Cyril} ROBERTO},\\
Laboratoire d'Analyse et Math\'ematiques Appliqu\'ees. UMR 8050, Universit\'es de Marne la Vall\'ee Paris Est et de Paris 12-Val-de-Marne\\
Boulevard Descartes, Cit\'e Descartes, Champs sur Marne\\
F-77454 Marne la Vall\'ee Cedex 2.} \email{cyril.roberto@univ-mlv.fr}

\thanks{The work of C.R. was partially supported by the European Research Council through the ``Advanced Grant''
PTRELSS 228032. P.C. and A.G. acknowledge the support of the ANR project EVOL.}

\maketitle
 \begin{center}

 \textsc{$^{\spadesuit}$  Universit\'e de Toulouse}
\smallskip

\textsc{$^{\diamondsuit}$ Universit\'e Blaise Pascal}
\smallskip

\textsc{$^{\clubsuit}$ Universit\'e Marne la Vall\'ee Paris Est}
\smallskip

 \end{center}

\begin{abstract}
We prove that for symmetric Markov processes of diffusion type admitting a ``carr\'e du champ'',
the Poincar\'e inequality is equivalent to the exponential convergence of the associated
semi-group in one (resp. all) $\L^p(\mu)$ spaces for $1<p<+\infty$. Part of this result extends to
the stationary non necessarily symmetric situation.
\end{abstract}
\bigskip

\bigskip

\bigskip

\section{Introduction and main results.}\label{Intro}

Let $X_t$ be a general Markov processes with infinitesimal generator $L$ and with state space some
Polish space $E$. We assume that the extended domain of the generator contains a nice core
$\mathcal D$ of uniformly continuous functions, containing the constant functions, which is an
algebra, for which we may define the ``carr\'e du champ'' operator
$$\Gamma(f,g) = \frac 12 \, \left(L(fg) - f Lg - gLf \right) \, .$$ Functions in $\mathcal D$ will
be called ``smooth''. The associated Dirichlet form can thus be calculated for smooth $f$'s as
$$\mathcal E(f,f) \, := \, - \, \int \, f \, Lf \, d\mu \, = \, \int \, \Gamma(f,f) \, d\mu \, .$$

In addition we assume that $L$ is $\mu$-symmetric for some probability measure defined on $E$.
Thus $L$ generates a $\mu$-symmetric (hence stationary) semi-group $P_t$, which is a contraction
semi-group on all $\L^p(\mu)$ for $1\leq p \leq +\infty$, and the $\L^2$ ergodic theorem (in the
symmetric case) tells us that for all $f\in \L^2(\mu)$,
$$\lim_{t \to +\infty} \, \parallel P_tf \, - \, \int f \, d\mu\parallel_{\L^2(\mu)} \, = \, 0 \,
.$$
For all this one can give a look at \cite{cat4}. Here and in the sequel, for any $p \in [1,\infty)$, $\parallel f \parallel_{\L^p(\mu)}$, or in a shorter way $\parallel f \parallel_p$,
stands for the $\mathbb{L}^p(\mu)$-norm of $f$ with respect to $\mu$: $\parallel f \parallel_p^p := \int |f|^p d\mu$.

\smallskip

It is then well known that the following two statements are equivalent
\begin{itemize}
\item[] (\textbf{H-Poinc}). \quad  $\mu$ satisfies a Poincar\'e inequality, i.e. there exists a
constant $C_P$ such that for all smooth $f$,
$$\Var_\mu(f):=\int f^2d\,\mu-\left(\int f\,d\mu\right)^2 \, \leq \, C_P \, \int \, \Gamma(f,f) \, d\mu \, .$$  \item[](\textbf{H-2}). \quad
There exists a constant $\lambda_2$ such that $$\Var_\mu(P_tf) \, \leq \, e^{- \, 2 \, \lambda_2
\, t} \, \Var_\mu(f) \, .$$
\end{itemize}
If one of these assumptions is satisfied we have $\lambda_2=1/C_P$.
\smallskip

In the sequel we shall assume in addition that $\Gamma$ comes from a derivation, i.e.
$$\Gamma(fg,h) = f \, \Gamma(g,h) + g \, \Gamma(f,h) \, ,$$ i.e. (in the terminology of
\cite{logsob}) that $X_.$ is a diffusion. We also recall the chain rule: if $\varphi$ is a $C^2$
function,
$$L(\varphi(f))=\varphi'(f) \, Lf + \varphi''(f) \, \Gamma(f,f) \, .$$
In this note we shall establish the following theorem

\begin{theorem}\label{thmmain}
For $f\in \L^p(\mu)$ define $N_p(f) \, := \,
\parallel f - \int f\, d\mu\parallel_p$. The following statements are equivalent
\begin{enumerate}
\item[(1)] \quad (\textbf{H-Poinc}) is satisfied, \item[(2)] \quad there exist some $1<p<+\infty$
and constants $\lambda_p$ and $K_p$ such that for all $f\in \L^p(\mu)$, $$N_p(P_tf) \, \leq \, K_p
\, e^{-\lambda_p \, t} \, N_p(f) \, ,$$ \item[(3)] \quad for all $1<p<+\infty$, there exist some
constants $\lambda_p$ and $K_p$ such that for all $f\in \L^p(\mu)$, $$N_p(P_tf) \, \leq \, K_p \,
e^{-\lambda_p \, t} \, N_p(f) \, .$$
\end{enumerate}
\end{theorem}
We shall denote by (\textbf{H-p}) the property (2) for a given $p$.
\medskip

Of course (3) implies (2). The fact that (2) implies (1) is a consequence of the following Lemma
which seems to be well known by the specialists in Statistical Physics (we learned this result
from P. Caputo and P. Dai Pra) and is used in \cite{Wu} (see lemma 2.6 therein). We shall give
however a very elementary proof in the next section.

\begin{lemma}\label{lemdecayf}
If it exists $\beta > 0$ such that for all $f\in \mathcal C$, $\mathcal C$ being an everywhere
dense subset of $\L^2(\mu)$, the following holds $\Var_\mu (P_t f) \leq c_f \, e^{-2 \, \beta \,
t} \, ,$ then $\Var_\mu (P_t f) \leq e^{-2 \, \beta \, t} \, \Var_\mu(f)$ for all $f \in
\L^2(\mu)$, i.e. the Poincar\'e inequality holds with $C_P \leq 1/\beta$.
\end{lemma}

An immediate consequence of Lemma \ref{lemdecayf} is that (2) implies (1) in the statement of
Theorem \ref{thmmain}.

Indeed if (\textbf{H-p}) holds for $p\geq 2$, $$N_2(P_t f) \leq N_p(P_t f) \leq K_p \,
e^{-\lambda_p \, t} \, N_p(f)$$ and applying Lemma \ref{lemdecayf} with $\mathcal C = \L^p(\mu)$
we deduce that $C_P \leq 1/\lambda_p$.

If (\textbf{H-p}) holds for $1<p\leq 2$ we may similarly write $$N_2(P_t f) \leq
N^{(2-p)/2}_\infty(f) \, N^{p/2}_p(f) \leq (K_p)^{p/2} \, e^{- \, p \lambda_p \, t/2} \,
N^{(2-p)/2}_\infty(f)  \, N^{p/2}_p(f)$$ and applying Lemma \ref{lemdecayf} with $\mathcal C =
\L^\infty(\mu)$ we deduce that $C_P \leq 2/ (p \, \lambda_p)$.
\medskip

Of course if $p$ and $q$ are conjugate exponents, (\textbf{H-p}) and (\textbf{H-q}) are
equivalent. More precisely we may write, with $\mu(f):=\int f\, d\mu$
\begin{eqnarray*}
|\int P_t(f-\mu(f)) \, g \, d\mu| & = & |\int P_t(f-\mu(f)) \, (g - \mu(g)) \, d\mu| = |\int
(f-\mu(f)) \, P_t(g - \mu(g)) \, d\mu| \\ & \leq & N_p(f) \, N_q(P_t g) \leq  \, K_q \, e^{ -
\lambda_q \, t} \, N_p(f) \, N_q(g) \\ & \leq & \, 2 \, K_q \, e^{ - \lambda_q \, t} \, N_p(f) \,
\parallel g\parallel_{\L^q(\mu)}
\end{eqnarray*}
if (\textbf{H-q}) holds. Hence
\begin{lemma}\label{lempq}
If $\frac 1p + \frac 1q =1$, (\textbf{H-q}) implies (\textbf{H-p}) with $K_p \leq 2 \, K_q$ and
$\lambda_p \geq \lambda_q$.

Accordingly if (\textbf{H-p}) holds, (\textbf{H-Poinc}) holds with $C_P \leq 1/\lambda_p$.
\end{lemma}
If $p\leq 2$ we obtain a better bound that the one we obtained directly.
\medskip

Lemma \ref{lempq} also shows that, in order to complete the proof of Theorem \ref{thmmain}, it is
enough to show that (1) implies (3) for all $p\geq 2$.
\medskip

Actually there are two interests in such a Theorem. The first one is obviously the rate of
convergence at infinity for which what is important is to get the largest possible $\lambda_p$,
despite the (reasonable) value of $K_p$. The second one is the opposite: get the result with
$K_p=1$ so that the inequality becomes an equality at time $t=0$ in order to possibly use the
result for isoperimetric controls for instance. The ideal situation is when we can reach these
goals simultaneously (as for $p=2$). As we shall see however, for $p>2$ we will obtain two results
described below.

\begin{theorem}\label{thmpetit}
If (\textbf{H-Poinc}) is satisfied, then for all $p>2$ (\textbf{H-p})  holds with $K_p=1$ and
$$\lambda_p \geq \frac{2^{k+6}}{(2^{7\times 2^{k+1}} \, C_P)} \quad \textrm{ if } \quad 2^{k+1}
\geq p
> 2^k \, \textrm{ for some } \, k>1 \, .$$

Consequently for $1<p<2$, (\textbf{H-p}) holds with $K_p=2$ and $\lambda_p=\lambda_{p/(p-1)}$.
\end{theorem}

Note that for $p=2$ we recover a worse constant that the known $\lambda_2=1/C_P$.
\medskip

We shall also prove

\begin{theorem}\label{thmgrand}
If (\textbf{H-Poinc}) is satisfied, then for all $p>2$ (\textbf{H-p})  holds with $\lambda_p=1/(p
\, C_P)$ and $K_p = 4^{1 - \frac 1p}$.

If $p=2^k$ for $k\geq 1$ one can improve these bounds in $\lambda_{2^k}=2/(2^k \, C_P)$ and $K_p =
4^{1- \frac 2p}$.

Consequently for $1<p<2$, (\textbf{H-p}) holds with $K_p=2$ and $\lambda_p=\lambda_{p/(p-1)}$.
\end{theorem}

Again we are loosing some factor (but here only $2$) for $p>2$ but close to $2$. Of course the
statements of both Theorem \ref{thmpetit} and Theorem \ref{thmgrand} indicate that the scheme of
proof will be to get the result for the successive powers of $2$ and then to interpolate between
them.
\medskip

The case $p=1$ is extensively studied in \cite{CatGui3} and the Poincar\'e inequality is no more
sufficient in general to obtain an exponential decay in $\L^1(\mu)$. Replacing $\L^p$ norms by
Orlicz norms (weaker than any $N_p$ for $p>1$) is possible provided one reinforces the Poincar\'e
inequality into a $F$-Sobolev inequality (see \cite{CatGui3} Theorem 3.1) as it is well known in
the case $F=\log$ for the Orlicz space $\L \, \log \, \L$.
\medskip

The question of exponential convergence in $\L^p$ ($p\neq 2$) was asked to us by M. Ledoux after a
conversation with A. Naor.
We did not find the statement of such a result in the literature. However recall that in
\cite{w04jfa}, F.Y. Wang used the equivalent Beckner type formulation of Poincar\'e inequality to
give a partial answer to the problem i.e., a Poincar\'e inequality with constant $C_P$ is
equivalent to the following: for any $1< p \leq 2$ and for any \underline{non-negative} $f$,
$$\int \, (P_tf)^p \, d\mu - \left(\int f\,d\mu\right)^p \leq e^{- \, \frac{4(p-1) \, t}{p \, C_P}} \, \left(\int \,
(f)^p \, d\mu - \left(\int f\,d\mu\right)^p \right) \, .$$ (One has to take care with the constants since some $2$
may or may not appear in the definition of $\Gamma$, depending on authors and of papers by the
same authors.) This result cannot be used to study the decay to the mean in $\L^p$ norm, but it is
of particular interest when studying densities of probability.

Note that the decay rate we obtain in Theorem \ref{thmgrand} is not comparable with the one in
Wang's result. Nevertheless, we recover here the $\L^1$ decay obtained in \cite{CatGui3} Example
2.3. so that, at least for powers of $2$, the rate obtained in Theorem \ref{thmgrand} seems to be
almost optimal.

\subsection*{Acknowledgement}
We warmly thank M. Ledoux for asking us about the exponential convergence in $\L^p$ ($p\neq 2$),
but also especially because he preciously saved a copy of one of our main arguments that we loosed
in our perfectly disordered office. The third author also would like to warmly thank Fabio
Martinelli and the University of Rome 3 where part of this work was done.

\bigskip

\section{Poincar\'e inequalities and $\L^p$ spaces.}\label{seclp}

We start with the \textbf{Proof of Lemma \ref{lemdecayf}. \quad}
\medskip

\begin{proof}
The proof lies on the following lemma proven in \cite{rw} using the spectral resolution
\begin{lemma}\label{lemlogsemi}
$t \mapsto \log \parallel P_t f\parallel_{\L^2(\mu)}$ is convex.
\end{lemma}
Here is a direct proof that does not use the spectral resolution. If $n(t)=\parallel P_t
f\parallel^2_{\L^2(\mu)}$, the sign of the second derivative of $\log n$ is the one of $n'' n -
(n')^2$. But $$n'(t)= 2 \, \int \, P_tf \, LP_t f \, d\mu$$ and
$$n''(t) = 2 \, \int \, (LP_tf)^2 \, d\mu + 2 \, \int \, P_tf \, LP_t Lf \, d\mu = 4 \, \int \,
(LP_tf)^2 \, d\mu \, ,$$ so that lemma \ref{lemlogsemi} is just a consequence of Cauchy-Schwarz
inequality.

In order to prove lemma \ref{lemdecayf}, assuming that $\int f d\mu=0$ which is not a restriction,
it is enough to look at
$$t \mapsto \log
\parallel P_t f\parallel_{\L^2(\mu)} + \beta \, t \, ,$$ which is convex, according to lemma
\ref{lemlogsemi}, and bounded since $\Var_\mu (P_t f) \leq c_f \, e^{- \, 2 \beta t}$. But a
bounded convex function on $\R^+$ is necessarily non-increasing. Hence $$\parallel P_t
f\parallel_{\L^2(\mu)} \leq e^{-\beta \, t} \, \parallel P_0 f\parallel_{\L^2(\mu)}$$ for all
$f\in \mathcal C$, the result follows using the density of $\mathcal C$.
\end{proof}
\bigskip

We come now to the proofs of our main theorems.

\textbf{Proof of Theorem \ref{thmpetit}.}
\medskip

\begin{proof}
The natural idea to study the time derivative of $N_p(P_tf)$, namely
$$\frac{d}{dt} N^p_p(P_tf) = p \, \int \, sign(P_t f - \mu(f)) \, |P_tf -\mu(f)|^{p-1} \, LP_tf \,
d\mu \, .$$ Hence we get an equivalence between
\begin{enumerate}
\item[] \quad There exists a constant $C(p)$ such that for all $f$,
\begin{equation}\label{eqconvp}
N_p^p(P_tf) \leq e^{- \, \frac{pt}{C(p)}} \, N_p^p(f) \, .
\end{equation}
 \item[] \quad There exists a constant $C(p)$
such that for all $f \in \mathcal D$ with $\mu(f)=0$,
\begin{equation}\label{eqpoingenebad}
N_p^p(f) \, \leq \, - \, C(p) \, \int \, sign(f) \, |f|^{p-1} \, Lf \, d\mu \, .
\end{equation}
\end{enumerate}
In order to compare all the inequalities \eqref{eqpoingenebad} to the Poincar\'e inequality (i.e.
$p=2$) one is tempted to make the change of function $f \mapsto sign(f) \, |f|^{2/p}$ (or $f
\mapsto sign(f) \, |f|^{p/2}$) and to use the chain rule. Unfortunately, first
$\varphi(u)=u^{2/p}$ is not $C^2$, second $\mu(sign(f) \, |f|^{2/p})\neq 0$ (the same for $p/2$
for the second argument).

However, for $p\geq 2$,  one can integrate by parts in \eqref{eqpoingenebad} which thus becomes
\begin{equation}\label{eqpoingenebadintegre}
N_p^p(f) \, \leq \, C(p) \, (p-1) \,\int \,  |f|^{p-2} \, \Gamma(f,f) \, d\mu \, = \, C(p) \,
\frac{4(p-1)}{p^2} \, \int \, \Gamma(|f|^{p/2},|f|^{p/2}) \, d\mu.
\end{equation}
\medskip

It thus remains to show that the Poincar\'e inequality implies \eqref{eqpoingenebadintegre} for
all $p\geq 2$. This will be done in two steps. First we will show the result for $p=4$. Hence
\eqref{eqconvp} hold for $p=2$ and $p=4$. According to the Riesz-Thorin interpolation theorem,
\eqref{eqconvp} (hence \eqref{eqpoingenebadintegre}) thus hold for all $2\leq p \leq 4$. Next we
shall show that if \eqref{eqpoingenebadintegre} holds for $p$ it holds for $2p$. This will
complete the proof by an induction argument. Of course the final step is the only necessary one
(starting with $p=2$) but we think that the details for $2p=4$ will help to follow the scheme of
proof for the general $2p$ case.
\smallskip

We proceed with the proof for $p=4$.

Assume that $\mu(f)=0$. First, applying the Poincar\'e inequality to $f^2$ we get $$\int f^4 d\mu
\leq \left(\int f^2 d\mu\right)^2 + 4 \, C_P \, \int f^2 \, \Gamma(f,f) \, d\mu \, ,$$ so that it
remains to prove that $$\left(\int f^2 d\mu\right)^2 \leq C \, \int f^2 \, \Gamma(f,f) \, d\mu \,
,$$ for some constant $C$.

Let now, for every $u>0$, $\varphi=\varphi_u : \R \mapsto \R$ be the $2$-Lipschitz function
defined by $\varphi(s)=0$ if $|s|\leq u$,  $\varphi(s)=s$ if $|s|\geq 2u$ and linear in between.
Applying Poincar\'e inequality to $\varphi(f)$ yields $$\int \, (\varphi(f))^2 \, d\mu \leq
\left(\int \, \varphi(f) \, d\mu\right)^2 + 4 \, C_P \, \int_{\{|f|\geq u\}} \, \Gamma(f,f) \,
d\mu \, .$$ But $$\int \,  (\varphi(f))^2 \, d\mu \geq \int_{\{|f|\geq 2u\}} \, f^2 \, d\mu \geq
\int f^2 d\mu \, - \, 4u^2 \, ,$$ and since $\mu(f)=0$, $$\left| \int \varphi(f) \, d\mu\right|
\leq 4u \, .$$ Summarizing, it follows that
\begin{eqnarray*}
\int f^2 \, d\mu & \leq & 20 u^2 + 4 \, C_P \, \int_{\{|f|\geq u\}} \, \Gamma(f,f) \, d\mu \\ &
\leq & 20 u^2 + \frac{4}{u^2} \, C_P \, \int \, f^2 \, \Gamma(f,f) \, d\mu \, .
\end{eqnarray*}
Optimizing in $u^2$ finally yields $$\left(\int f^2 \, d\mu\right)^2 \leq 320 \, C_P \, \int \,
f^2 \, \Gamma(f,f) \, d\mu \, ,$$ i.e. $$N_4^4(f) \, \leq \, 324 \, C_P \, \int \, f^2 \,
\Gamma(f,f) \, d\mu \, .$$ The constant $324$ is of course not optimal, but replacing the $2$ by
$2a$ in the definition of $\varphi$ yields of course the same constant.
\smallskip

Now assume that \eqref{eqpoingenebadintegre} holds for some $p\geq 2$ and of course the Poincar\'e
inequality holds with constant $C_P$. First we apply Poincar\'e inequality to the function
$|f|^p$,
$$\int |f|^{2p} d\mu \leq  \left(\int |f|^p \, d\mu\right)^2 + C_P \, p^2 \, \int \, |f|^{2p-2} \,
\Gamma(f,f) \, d\mu \, .$$

Now as in the previous step we introduce $\varphi$ and remark that $$\int |f|^p d\mu \leq \int \,
|\varphi(f)|^p \, d\mu + 2^p \, u^p \, .$$  We write \eqref{eqpoingenebadintegre} for the function
$\varphi(f) -\mu(\varphi(f))$ and then apply  $|a+b|^{q} \leq 2^{q-1} \, (|a|^q + |b|^q)$ for
$q\geq 1$ and $|a+b|^{q} \leq 2^{q} \, (|a|^q + |b|^q)$ if $q\geq 0$, and recalling that
$|\mu(\varphi(f))|\leq 4u$ in order to obtain
\begin{eqnarray*}
\int |\varphi(f)|^{p} d\mu &\leq& 2^{p-1} \, \left((p-1) \, C(p) \, \int \, |\varphi(f)
-\mu(\varphi(f))|^{p-2} \, \Gamma(\varphi(f),\varphi(f)) \,  d\mu + |\mu(\varphi(f))|^p\right) \\
&\leq& 2^{p-1} \, (p-1) \, C(p) \, 4 \, \int_{\{|f|\geq u\}} \,
2^{p-2}\left(|f|^{p-2}+|\mu(\varphi(f))|^{p-2}\right) \, \Gamma(f,f) \, d\mu \\ & &  + 2^{p-1}
|\mu(\varphi(f))|^p \\ & \leq & 2^{2p-1} \, (p-1) \, C(p) \,  \int_{\{|f|\geq u\}} \, |f|^{p-2} \,
\frac{|f|^p}{u^p} \, \Gamma(f,f) \, d\mu \\ & & + 2^{4p-5} \, (p-1) \, C(p) \, u^{p-2} \,
\int_{\{|f|\geq u\}} \, \frac{|f|^{2p-2}}{u^{2p-2}} \, \Gamma(f,f) \, d\mu \\ & & + 2^{3p-1} \,
u^{p}\\ &\leq& 2^{3p-1} \, u^{p} + (2^{2p-1}+2^{4p-5}) \, \frac{C(p) \, (p-1)}{u^p} \left(\int \,
|f|^{2p-2} \, \Gamma(f,f) \, d\mu\right) \, .
\end{eqnarray*}
Again we optimize in $u^p$ and obtain $$\left(\int |f|^{p} d\mu\right)^2 \leq 4 \,
(2^{2p-1}+2^{4p-5})(2^p+2^{3p-1}) \, (p-1) \, C(p) \, \left(\int \, |f|^{2p-2} \, \Gamma(f,f) \,
d\mu\right) \, ,$$ and finally $$\int |f|^{2p} d\mu \leq \left(4 \,
(2^{2p-1}+2^{4p-5})(2^p+2^{3p-1}) \, (p-1) \, C(p) + p^2 \, C_P\right) \, \int \, |f|^{2p-2} \,
\Gamma(f,f) \, d\mu \, ,$$ and the proof is completed.
\end{proof}
\smallskip

Of course the final step is available for $p=2$ and $C(2)=C_P$ but it furnishes a still worse
constant than $324 C_P$. The value of $\lambda_p$ for $p=2^k$ can be obtained by induction.
\bigskip

\textbf{Proof of Theorem \ref{thmgrand}}
\medskip

\begin{proof}
We shall prove by induction that, provided (\textbf{H-Poinc}) holds, the following holds true for
all $k\geq 1$: if $p=2^k$, for all $t\geq 0$
\begin{equation}\label{eqdecayciril}
N_p^p(P_tf) \leq 4^{p-2} \, e^{ - 2t/C_P} \, N_p^p(f) \, .
\end{equation}
For $k=1$ (i.e $p=2$) \eqref{eqdecayciril} is equivalent to (\textbf{H-Poinc}).

Now we proceed by induction. Without loss of generality we assume that $\int f \, d\mu = 0$ and
denote by $U_k(t) := N_p^p(P_t f)$ for $p=2^k$. Recall that
\begin{eqnarray*}
U'_k(t) &=& 2^k \, \int \, sign(P_tf) \, |P_tf|^{p-1} \, LP_t f \, d\mu \\ &=&  - 2^k \, (2^k-1)
\, \int \, (P_t f)^{2^k-2} \, \Gamma(P_tf,P_tf) \, d\mu \\ &=& - \, 4 \, (2^k -1) \, 2^{-k} \,
\int \, \Gamma((P_tf)^{2^{k-1}},(P_tf)^{2^{k-1}}) \, d\mu \\ &\leq& - \, 3 \, \int \,
\Gamma((P_tf)^{2^{k-1}},(P_tf)^{2^{k-1}}) \, d\mu \, ,
\end{eqnarray*}
since for $k\geq 1$, $3\leq 4 \, (2^k -1) \, 2^{-k}$. In addition the Poincar\'e inequality
applied to $(P_tf)^{2^{k-1}}$ yields $$U_k(t) \leq U^2_{k-1}(t) + C_P \, \int \,
\Gamma((P_tf)^{2^{k-1}},(P_tf)^{2^{k-1}}) \, d\mu \, .$$ Putting these inequalities together we
thus have
\begin{equation}\label{eqrecure}
U'_k(t) \leq - \frac{3}{C_P} \, U_k(t) + \frac{3}{C_P} \, U^2_{k-1}(t) \, .
\end{equation}
We may thus apply Gronwall's lemma and obtain $$U_k(t) \leq e^{- \, 3t/C_P} \, \left(U_k(0) +
\frac{3}{C_P} \, \int_0^t \, e^{3s/C_P} \, U^2_{k-1}(s) \, ds\right) \, .$$ If
\eqref{eqdecayciril} holds for $p=2^{k-1}$ with $k-1 \geq 1$, we thus obtain
\begin{eqnarray*}
U_k(t) &\leq& e^{- \, 3t/C_P} \, \left(U_k(0) + \frac{3}{C_P} \, \int_0^t \, e^{3s/C_P} \,
\left(4^{2^{k-1}-2} e^{- \, 2s/C_P} \,  U_{k-1}(0)\right)^2 \, ds\right) \,\\ &\leq& e^{- \,
3t/C_P} \, \left(U_k(0) + 3 \, 4^{2^{k}-4} \, U^2_{k-1}(0) \, \int_0^t \, \frac{1}{C_P} \, e^{- \,
s/C_P} \, ds\right) \\ &\leq& e^{- \, 3t/C_P} \, \left(U_k(0) + 3 \, 4^{2^{k}-4} \, U^2_{k-1}(0)
\right)\\ &\leq& e^{- \, 2t/C_P} \, U_k(0) \, \left(1 + 3 \, 4^{2^{k}-4}\right) \, ,
\end{eqnarray*}
since $U_{k-1}^2(0) \leq U_k(0)$ thanks to Cauchy-Schwarz inequality.

Finally remark that $4^{2^k -2} \geq 1 + 3 \times 4^{2^k - 4}$ for $k\geq 2$, so that the
induction is completed.

Hence \eqref{eqdecayciril} is true for all $p=2^k$. In order to apply again the Riesz-Thorin
interpolation theorem for $2^k < p \leq 2^{k+1}$ and complete the proof of the theorem it remains
to note that $$4^{1- \frac 1p} \, e^{ - t/p\,C_P} \geq \max \left(4^{\frac{2^k -2}{2^k}} \, e^{-
2t/2^k \, C_P} \, ; \, 4^{\frac{2^{k+1} -2}{2^{k+1}}} \, e^{- 2t/2^{k+1} \, C_P}\right) \, .$$
\end{proof}
\bigskip

\section{Another proof of Theorem \ref{thmmain}.}

Let us start with a remark
\begin{remark}
Using H\"{o}lder inequality we see that \eqref{eqpoingenebadintegre} implies that for
$$\kappa(p)=(C(p) \, (p-1))^{p/2} \, ,$$
\begin{equation}\label{eqpoincp}
N_p^p(f) \, \leq \, \kappa(p) \,\int \,  \Gamma^{p/2}(f,f) \, d\mu \, .
\end{equation}
The latter is a $\L^p$ Poincar\'e inequality which was used in \cite{DGGW} and particularly
studied in \cite{Emil1}.

As recalled by E. Milman, we can replace the mean $\mu(f)=\int f\, d\mu$ by a median $m_\mu(f)$ in \eqref{eqpoincp}. Indeed
according to Lemma 2.1 in \cite{Emil1}, for all $1\leq p <+\infty$
\begin{equation}\label{eqcompmilman}
\frac 12 \, N_p(f) \leq \parallel f - m_\mu(f)\parallel_p \leq 3 \, N_p(f) \, .
\end{equation}
Hence up to the constants we may replace $\mu(f)=0$ by $m_\mu(f)=0$ in \eqref{eqpoincp}. Now the
transformations $f \mapsto sign(f) \, |f|^{h}$ with $h=2/p$ or $h=p/2$ is preserving the fact that
$0$ is a median so that we easily obtain (see \cite{Emil1} Proposition 2.5)
\begin{proposition}\label{propcomppoincp}
If $\mu$ satisfies \eqref{eqpoincp} for some $p_0\geq 1$ with a constant $\kappa(p_0)$, then it
satisfies \eqref{eqpoincp} for all $p\geq p_0$, with a constant $$\kappa(p) \leq \left(\frac{6p}
{p_0}\right)^p \, \kappa^{p/p_0}(p_0).$$
\end{proposition}
\hfill $\diamondsuit$
\end{remark}

Unfortunately the same reasoning fails with \eqref{eqpoingenebadintegre} since there is no obvious
comparison between $ \int \, |f-\mu(f)|^{p-2} \, \Gamma(f,f) \, d\mu$ and $ \int \,
|f-m_\mu(f)|^{p-2} \, \Gamma(f,f) \, d\mu$.

However we shall see that one can nevertheless use the median in order to prove Theorem
\ref{thmmain}, but that doing so furnishes disastrous constants.
\medskip

Introduce some new notation. If $f\in \L^p$, denote by $M_p^p(f)=\int |f -m_\mu(f)|^p \, d\mu$ and
the new inequality
\begin{equation}\label{eqlpmed}
M_p^p(f) \leq B(p) \, \int \, |f-m_\mu(f)|^{p-2} \, \Gamma(f,f) \, d\mu \, .
\end{equation}
we then have

\begin{theorem}\label{thmsurprisemean}
All the inequalities \eqref{eqlpmed} are equivalent (for $+\infty>p\geq 2$ of course). Furthermore
the best constants $B(p)$ satisfy $B(p) = \frac{p^2}{4} \, B(2)$.
\end{theorem}
\begin{proof}
Let $f$ with $m_\mu(f)=0$. If \eqref{eqlpmed} holds for $p=2$ (i.e. the Poincar\'e inequality
holds thanks to \eqref{eqcompmilman}), we apply it with $g= sign(f) \, |f|^{p/2}$ and get $$\int
|f|^p \, d\mu = \int g^2 \, d\mu \leq B(2) \, \frac{p^2}{4} \, \int \, |f|^{p-2} \, \Gamma(f,f) \,
d\mu \, ,$$ i.e. \eqref{eqlpmed} holds for $p$ with $B(p)\leq \frac{p^2}{4} \, B(2)$.

Conversely if \eqref{eqlpmed} holds for some $p\geq 2$, we apply it with the function $$g =
sign(f) \, |f|^{2/p} \, \BBone_{|f|\geq s} + s^{\frac{2-p}{p}} \, f \, \BBone_{|f|\leq s} \, ,$$
defined for $s>0$. We thus obtain
$$\int_{|f|\geq s} \,  |f|^2 d\mu + s^{2-p} \, \int_{|f|< s} \, |f|^p \,
d\mu = \int |g|^p \, d\mu$$ and
\begin{eqnarray*}
\int |g|^p \, d\mu  & \leq & B(p) \left(\int_{|f|\geq s} \, |f|^{\frac 2p \, (p-2)} \,
\frac{4}{p^2} \, \, |f|^{\frac{2(2-p)}{p}} \, \Gamma(f,f) \, d\mu \right) + \\ & & + \, B(p) \,
\left(s^{2-p} \, \int_{|f|< s} \, |f|^{p-2} \, \Gamma(f,f) \, d\mu\right)\\ & \leq & B(p) \,
\left(\frac{4}{p^2} \, \int_{|f|\geq s} \, \Gamma(f,f) \, d\mu + \int_{|f|< s} \, \Gamma(f,f) \,
d\mu\right) \, ,
\end{eqnarray*}
so that by letting $s$ go to $0$ we obtain $B(2) \leq \frac{4}{p^2} \, B(p)$, hence the result.
\end{proof}

We shall now see how to use this theorem in order to study $N_p(P_tf)$.
\medskip

First recall that, according to \eqref{eqcompmilman}, \eqref{eqlpmed} for $p=2$ is equivalent to
the Poincar\'e inequality, and $$\frac 19 \, B(2) \, \leq \, C_P \, \leq \, 4 \, B(2) \, .$$ It
follows, using again \eqref{eqcompmilman} and Theorem \ref{thmsurprisemean}, that if
(\textbf{H-Poinc}) holds, for any $p\geq 2$,
\begin{eqnarray}\label{arraycomp}
2^{-p} \, N_p^p(f) & \leq & M_p^p(f) \, \leq \,  B(p) \, \int \, |f-m_\mu(f)|^{p-2} \, \Gamma(f,f)
\, d\mu  \\ & \leq & B(p) \, \delta(p-2) \, \int \, |f-\mu(f)|^{p-2} \, \Gamma(f,f) \, d\mu
\, + \nonumber \\ & + & B(p) \, \delta(p-2) \,  |\mu(f) - m_\mu(f)|^{p-2} \, \int \, \Gamma(f,f) \, d\mu \nonumber \\
& \leq & B(p) \, \delta(p-2) \, \int \, |f-\mu(f)|^{p-2} \, \Gamma(f,f) \, d\mu \, +  \nonumber
\\ & + & B(p) \, \delta(p-2)\, 2^{(p-2)/2} \, (\Var_\mu(f))^{(p-2)/2} \, \int \, \Gamma(f,f) \, d\mu \, , \nonumber
\end{eqnarray}
where we have used $$ |\mu(f) - m_\mu(f)| \leq \sqrt{2} \, (\Var_\mu(f))^{1/2}$$ (see the proof of
Lemma 2.1 in \cite{Emil1}) and $$(u+v)^p \leq \delta(p) (u^p+v^p)$$ for any non-negative $u$ and
$v$ and any $p\geq 0$, with $\delta(p)=2^{p-1}$ if $p \geq 1$ and $\delta(p)=1$ if $0\leq p \leq
1$; hence finally $\delta(p)=1\vee 2^{p-1}$.
\medskip

Now consider, for $p\geq 2$, the following ``entropy functional'' (in the terminology of P.D.E.
specialists)
\begin{equation}\label{eqfunc}
E_p(f) = a_p \, N_p^p(f) + b_p \, (\Var_\mu(f))^{p/2} \, \leq \, (a_p+b_p) \, N_p^p(f) \, ,
\end{equation}
where $a_p$ and $b_p$ are positive constants to be chosen later. Remark first that using
Poincar\'e inequality and \eqref{arraycomp}, we have (remember $B(p)\le 9\, C_P \, p^2/4$)
\begin{equation}\label{simpoinc}
E_p(f) \leq A(p) + D(p) \, ,
\end{equation}
where
\begin{eqnarray*}
A(p)&=& a_p\, 2^p\,\frac{9C_P p^2}4\,\delta(p-2) \, \int \, |f-\mu(f)|^{p-2} \, \Gamma(f,f) \,
d\mu \, ,
\\D(p)&=& \left(a_p\, 2^p\, \frac{9C_P p^2}4\, 2^{(p-2)/2}\,\delta(p-2)+C_P
b_p\right)\,(\Var_\mu(f))^{(p-2)/2} \, \int \, \Gamma(f,f) \, d\mu.
\end{eqnarray*}
We have
\begin{eqnarray*}
\frac{d}{dt} E_p(P_tf) &=& - \, a_p \, p (p-1)\, \int \, |P_tf-\mu(P_tf)|^{p-2} \,
\Gamma(P_tf,P_tf) \, d\mu \,
\\ &&-
\, b_p \, \frac{p}{2} \, (\Var_\mu(P_tf))^{(p-2)/2} \, \int \, \Gamma(P_tf,P_tf) \, d\mu \, ,\nonumber\\
&=& -\frac{p(p-1)}{\frac{9C_P p^2}4\,\delta(p-2) 2^p}\, a_p\, 2^p\,\frac{9C_P p^2}4\,\delta(p-2) \,
\int \, |P_tf-\mu(f)|^{p-2} \, \Gamma(P_tf,P_tf) \, d\mu\\&&-\frac{b_p\frac{p}{2}}{a_p\, 2^p\, \frac{9C_P p^2}4\,
2^{(p-2)/2}\,\delta(p-2)+C_P b_p}\,\left(a_p\, 2^p\, \frac{9C_P p^2}4\, 2^{(p-2)/2}\,\delta(p-2)+C_P b_p\right)\,\\
&&\qquad\times(\Var_\mu(P_tf))^{(p-2)/2} \, \int \, \Gamma(P_tf,P_tf) \, d\mu \, . \nonumber
\end{eqnarray*}
Using \eqref{simpoinc}, and choosing $a_p,b_p$ such that
$$\frac{p(p-1)}{\frac{9C_P p^2}4\,\delta(p-2) 2^p}=\frac{b_p\frac{p}{2}}{a_p\, 2^p\, \frac{9C_P p^2}4\, 2^{(p-2)/2}\,\delta(p-2)+C_P b_p}$$
which is possible as $p\ge 2$, we thus get
$$\frac{d}{dt} E_p(P_tf)\le -\frac{p(p-1)}{\frac{9C_P p^2}4\,\delta(p-2) 2^p} E_p(P_tf)=-\gamma_p E_p(P_tf)$$
so that applying Gronwall's lemma we deduce $$N_p^p(P_tf) \leq \frac{1}{a_p} \, E_p(P_tf) \leq
e^{- \, \gamma_p \, t} \, E_p(f) \leq \frac{a_p+b_p}{a_p} \, e^{- \, \gamma_p \, t} \, N^p_p(f) \,
.$$ Putting all our results together, we have thus shown:

\begin{theorem}\label{thmavecmed}
If (\textbf{H-Poinc}) holds with constant $C_P$, then for all $p\geq 2$, $$N_p(P_tf) \leq K_p \,
e^{- \, \lambda_p \, t} \, N_p(f) \, ,$$ with $$\lambda_p =\frac{4(p-1)}{9\,p^2\, (1\vee 2^{p-3})
\, 2^p\,C_P}
 \, ,$$ and
 $$K_p^p=1+ \frac{\frac{9 p^2}4\,(p-1)2^{(3p-2)/2}\, (1\vee 2^{p-3})}{\frac{9 p^2}4\,(1\vee 2^{p-3}) \, 2^{p-1}-p+1} \, .$$
\end{theorem}

In this result, the constant $K_p$ is, at least for large $p$, smaller than the one obtained in Theorem
\ref{thmgrand} but of course the constant $\lambda_p$ is quite bad, but however better than the one in Theorem \ref{thmpetit}.

\bigskip

\section{Some final Remarks.}

We did not succeed in proving the analogue of Lemma \ref{lemdecayf} for $p>2$ (and actually we
believe that such a statement is false). Hence both Theorems \ref{thmpetit}, \ref{thmgrand}  and \ref{thmavecmed}
have their own interest.

Of course under stronger assumptions than the sole Poincar\'e inequality (logarithmic Sobolev
inequality for instance), one can improve the bounds obtained in Theorem \ref{thmmain}.
\medskip

\textbf{Extension to the non-symmetric case.}

Notice that the only point where we used symmetry is the proof of Lemma \ref{lemlogsemi}, hence of
Lemma \ref{lemdecayf}. In particular if $\mu$ is invariant but not necessarily symmetric,
(\textbf{H-Poinc}) implies exponential decay in all the $\L^p(\mu)$, $p\geq 2$, and our bounds are
available, in particular we may choose $K_p=1$.

But if $(\textbf{H-p})$ holds for some $p>2$ and with $K_p=1$ (which is crucial) then
\eqref{eqpoingenebadintegre} is satisfied, which in return implies the same decay for the dual
semi-group $P_t^*$. Hence the duality argument shows that $(\textbf{H-q})$ is satisfied for both
$P_t$ and $P_t^*$, where $q$ is the conjugate exponent of $p$. Hence (\textbf{H-Poinc}) implies
exponential decay in all the $\L^p(\mu)$, $1<p<+\infty$.
\medskip

Conversely, assume that $(\textbf{H-p})$ holds for some $p>2$ and with $K_p=1$ (which is still
crucial). The previous argument shows that $(\textbf{H-q})$ is satisfied. The Riesz-Thorin
interpolation theorem then shows that $(\textbf{H-s})$ is satisfied for all $q\leq s \leq p$,
hence for $s=2$. But since we do not know that $K_2=1$, we cannot conclude that the Poincar\'e
inequality is satisfied. Also note that the induction argument we used in the  proofs calls
explicitly upon the Poincar\'e inequality, so that we cannot deduce that $(\textbf{H-s})$ holds
for $s>p$.
\medskip

Finally recall that in the non-symmetric situation, exponential decay in $\L^2$ can occur while
the Poincar\'e inequality is not satisfied. Of course in this situation, $K_2 >1$. This is the
generic situation in many hypocoercive kinetic models like the kinetic Ornstein-Uhlenbeck process
studied in \cite{vil} (also see \cite{BCG} section 6).
\bigskip

\bigskip

\bibliographystyle{plain}

\def\cprime{$'$}

\end{document}